\documentclass[11pt]{article}
\usepackage{times}
\usepackage{amsmath,amssymb}
\usepackage{theoremref}
\usepackage[utf8]{inputenc}
\usepackage{indentfirst}
\usepackage{graphicx,tikz}
\usepackage{xcolor}
\usepackage{xskak}
\usepackage[T1]{fontenc}
\usepackage{amsthm}
\newtheorem{problem}{Question}

\newtheorem{theorem}[problem]{Theorem}

\newtheorem{lemma}[problem]{Lemma}
\newtheorem{conjecture}[problem]{Conjecture}

\title{Greedy Gossiping}
\author{Kada K Williams}

\begin{document}

\maketitle

\begin{abstract}
The renowned Gossiping Problem (1971) asks the following. There are $n$ people who each know an item of gossip. In a telephone call, two people share all the gossip they know. How many calls are needed for all of them to be informed of all the gossip? If $n\ge 4$, the answer is $2n-4$.

We initiate and solve the related Greedy Gossiping Problem: given a fixed number $m<2n-4$ of calls, at most how much gossip can be known altogether? If every call increases the total knowledge of gossip as much as possible, the sum reaches $n^2$ only when $m=2n-3$. Our main result is that surprisingly, for each $m<2n-4$, this calling strategy is optimal.
\end{abstract}

\section{Introduction}

In 1971, Boyd asked the following question \cite{HMS}. There are $n$ people who each know an item of gossip. In a telephone call, two people share all the gossip they know. How many calls are needed for all of them to be informed of all the gossip?

This problem, commonly referred to as the 'Gossiping Problem', was independently resolved by Tijdeman \cite{Tij}, Baker and Shostak \cite{BSh}, as well as Hajnal, Milner, and Szemerédi \cite{HMS}. Possibly due to the simplicity of and variety of approaches to the Gossiping Problem, Bollobás additionally introduced it as a challenging 'starred question' for undergraduate students \cite{Bol}.

In this note, we tackle the related question of what happens when the number of calls is restricted to $m$ calls. The dons call greedily in the sense that if each don counted how much gossip she knows, the sum total $N$ would be largest possible. We characterise the optimal value of $N$ as follows.

\begin{theorem} \label{main}
Let $0\le m\le 2n-5$. In Gossiping, suppose that after $m$ calls, the number of news known is $N$ altogether. The maximal value of $N$ is achieved by the first $m$ of the following $2n-3$ calls: dons $2,3,\dots,n$ call don $1$, then don $1$ calls dons $2,3,4,\dots,n-1$.
\end{theorem}

In our proof, the range $0\le m\le n-1$ is easy to resolve, $n\le m\le 2n-6$ requires ideas from \cite{BSh}, and $m=2n-5$ is more difficult. For an example of $m=2n-4$ calls after which $N=n^2$ ($n\ge 4$), we refer the reader to \cite{Tij}, \cite{BSh}, \cite{HMS}.

Incidentally, a plethora of research is available involving ideas related to the original 'gossiping dons' problem. This includes questions about broadcasting \cite{HHL}, perpetual gossip \cite{Sco}, shuffles \cite{JJL}, and counting algorithms \cite{EMM}. These variations apply respectively to situations where only one piece of news is to be disseminated across an entire network, where the dons continually update one another, where information is replaced with an amount of a resource, or where information overload is to be prevented.

\section{Proof of Theorem \ref{main}}

\subsection{Distinction between $0\le m\le n-1$ and $n-1\le m\le 2n-5$}

Suppose dons $v_1,v_2,\dots,v_n$ make the calls $e_1,e_2,\dots,e_m$ in this order. The vertex set $\{v_1,\dots,v_n\}$ defines a {\it temporal graph} on this multiset of edges, where we say that $e_t$ occurs at time $t$ \cite{JJL}. Notice that don $v_1$ hears the gossip of $v_2$ if and only if there is a path of edges at increasing times, a {\it temporal path}, from $v_2$ to $v_1$.

Let us consider the call sequence $\epsilon_t=\{v_1,v_{t+1}\}$ for $1\le t\le n-1$ and then $\epsilon_{n-1+t}=\{v_1,v_{t+1}\}$ for $1\le t\le n-1$. In the first $n-1$ calls, $\epsilon_t$ lets $v_{t+1}$ learn the news from $v_1,\dots,v_t$ and $v_1$ learn the news from $v_{t+1}$, $1\le t\le n-1$. In the last $n-1$ calls, $v_1$ knows everything, and $\epsilon_{n-1+t}$ lets $v_{t+1}$ learn the news from $v_{t+2},\dots,v_n$. This much is clear by inducting on $t$.

We are claiming that $e_t=\epsilon_t$ ($1\le t\le m$) is optimal, with total gossip knowledge
$$N=\begin{cases} n+(2+3+\dots+(m+1)) \quad \text{if} \quad 1\le m\le n-1 \\ n^2-(1+2+\dots+(m'-1)) \quad \text{if}\quad m=2n-2-m',1\le m'\le n-1\end{cases}$$
since if $m\ge n-1$, after $e_1,e_2,\dots,e_{n-1}$, $v_{t+1}$ would be yet to learn $n-t-1$ new gossips ($1\le t\le n-1$) before $N=n^2$.

\subsection{The case $0\le m\le n-1$}

\begin{lemma}
At most $t+1$ pieces of news are communicated in call $e_t$ ($1\le t<n$).
\end{lemma}

\begin{proof}
Consider the simple graph on the set of edges $\{e_1,e_2,\dots,e_t\}$. The news that could be spoken of in call $e_t$ can only be from a vertex in the connected component of $e_t$, which contains at most $t+1$ vertices, if connected.
\end{proof}

Summing over $t=1,\dots,m$ yields the claimed result when $0\le m\le n-1$.

\subsection{No don hears her own gossip}

\begin{lemma}
Suppose there are $n+1$ dons whose calls include a subsequence of times $t_0<t_1<\dots<t_j$ when $v_{n+1}$ calls $v_{i_1}$, who calls $v_{i_2}$, and so on, until $v_{i_j}$ calls $v_{n+1}$. Then it is possible to arrange $2$ fewer calls among the $n$ dons, excepting $v_{n+1}$, such that $v_1,\dots,v_n$ communicate at least as much. \cite{BSh}
\end{lemma}

\begin{proof}
It is always possible to arrange $1$ fewer calls: if $v_{n+1}$ calls $u_1,u_2,\dots,u_s$ in this order, then instead, let $u_r$ call $u_{r+1}$ ($1\le r<s$) at the time she would call $v_{n+1}$. This way, transmission via $v_{n+1}$ is replaced by transmission along the temporal path $u_1u_2\dots u_s$.

However, if between the calls of $v_{n+1}$ with $u_p$ and $u_q$ ($p<q$), there is a temporal path from $u_p$ to $u_q$, then the calls of $u_p,\dots,u_{q-1}$ can be replaced as follows. Whenever $p<r<q$, let $u_r$ call the don currently at the end of the path when she would call $v_{n+1}$.

Therefore, if $v_{n+1}$ hears her own information, the other $n$ dons can communicate to the same effect, while sparing a call for $u_s$ and $u_p$.   
\end{proof}

Thus, we deduce that if there are $2n-m'$ calls among $n+1$ dons with $N>(n+1)^2-(1+2+\dots+(m'-1))$, and someone hears back her news, then there are $2n-2-m'$ calls among $n$ dons with $N>n^2-(1+2+\dots+(m'-1))$, because that don and her news contributes to $N$ by at most $(n+1)+n=(n+1)^2-n^2$. 

It is known that $m'= n$ cannot occur, and if $m'<n$, then we can pass to a smaller counterexample $n$. Therefore, in a counterexample with minimal $n$, there is no don who hears her own gossip.

\subsection{The case $n\le m\le 2n-5$}

\begin{lemma}
Suppose that $v_i$ does not hear her own gossip. Let $a_i$ be the number of dons who know the gossip of $v_i$ and let $v_i$ know $b_i$ gossips. Further, let $c_i$ be the number of calls that neither inform $v_i$ nor pass on the gossip of $v_i$, and let $d_i$ be the number of calls with $v_i$. Then
$$a_i+b_i\le m+2+d_i-c_i.$$
\end{lemma}

\begin{proof}
Highlight all the calls where $v_i$'s gossip is news. There are at least $a_i-1$ such calls, of which at least $a_i-1-d_i$ are not a direct call with $v_i$. Since $v_i$ never hears her own gossip, such a call cannot be in a temporal path to $v_i$, so these calls are not in the temporal paths that would inform $v_i$ with $b_i-1$ other gossips. Now
$$(a_i-1-d_i)+(b_i-1)+c_i\le m,$$
which rearranges to what was claimed.
\end{proof}

The gossiping number is $\sum_{i=1}^n a_i=N=\sum_{i=1}^n b_i$, while $\sum_{i=1}^n d_i=2m$. Summing both sides of the Lemma therefore yields $$2N\le n(m+2)+2m.$$
Subtracting $n(m+2)+2m$ from $2n^2-2(1+2+\dots+(m'-1))$ yields
$$-m'^2+(n+3)m'-4(n+1)=(m'-4)(n+1-m'),$$
which is non-negative in the range $4\le m'\le n-1$. Therefore, we have proven Theorem \ref{main} whenever $m\le 2n-6$, and from now on we shall consider $m=2n-5$.

\subsection{The subcase $m=2n-5$}

\begin{lemma}
After the first $n-2$ calls, either there is a don who has not been in a call, or the calls can be permuted without changing the knowledge gained in each so that for $1\le i\le n-3$, $e_{n-2+i}$ is not disjoint from all of $e_{n-1},e_n,\dots,e_{n-3+i}$.
\end{lemma}

\begin{proof}
The key idea here is to swap consecutive calls that are disjoint \cite{KSh}. If $e_{n-2+i}$ were disjoint from every call since $e_{n-2}$, it could replace $e_{n-2}$ after $i$ such swaps. 

At least two of the connected components in the graph with edges $e_1,\dots,e_{n-2}$ are trees, and if there is an isolated vertex, we are done. Otherwise, after making swaps, we may suppose that $e_{n-2}$ is in the smallest of the trees and $e_{n-3}$ is in the next smallest. If a later edge can replace $e_{n-2}$, this decreases the size of the smallest tree or else maintains it, so both endpoints of that edge in the tree. From thence, it can replace $e_{n-3}$, decreasing the size of the next smallest or smallest tree. The claim follows by infinite descent.
\end{proof}

In the case where a don $v_i$ has not called until $e_{n-2}$, at most $n-2$ dons learn her gossip, and so $a_i\le n-2$. In the case where the last $n-2$ calls are connected as a graph, there is an isolated vertex $v_i$ that only learns from the first $n-3$ calls, and so $b_i\le n-2$.

By reversing the calls $e_1,\dots,e_m$ to $e_m,\dots,e_1$, the roles of $a_i$ and $b_i$ exchange. Hence, without loss of generality, $a_i\le n-2$ for some $i$. For other values of $i$, $a_i\le n$, so $N\le n^2-2$. Our aim is to demonstrate $N\le n^2-3$ by ruling out the case of equality: $a_i=n$ for all $i$ except one with $a_i=n-2$.

\begin{lemma}
For each don, look at who they called first. If, say, the call $\{v_1,v_2\}$ was the first for $v_1$, but $\{v_2,v_3\}$ was the first for $v_2$, then $v_3$ cannot hear the news from $v_1$, whence $a_1\neq n$.
\end{lemma}

\begin{proof}
There can only be one call between any two dons, unless one of them hears their own gossip. Now if there were a temporal path from $v_1$ to $v_3$, as it begins with $\{v_1,v_2\}$ or after it, we would find a temporal path from $v_2$ to $v_3$ after the call $\{v_2,v_3\}$. However, $v_3$ does not hear her own gossip.
\end{proof}

If a call is the first for both callers, we say it is a mutual first call, and similarly for final calls. Notice that if $\{v_1,v_2\}$ is a mutual first call, then afterwards, whoever knows $v_1$'s gossip also knows $v_2$'s gossip. Hence, for the $i$ with $a_i=n-2$, the first call is not mutual. We also know that if $a_i=n$, the first call is mutual. Thus, there are $\frac{n-1}{2}$ mutual first calls. As for mutual final calls, we would be done if there were three distinct vertices whose final call is not mutual, with $b_i\le n-1$ thrice. Hence, for parity reasons, $\frac{n-1}{2}$ final calls are mutual. The mutual first and mutual final calls are different, else two dons would be disconnected from the others, whence there are $n-4$ miscellaneous calls.

\begin{lemma}
Suppose there are $\frac{n-1}{2}$ mutual first calls, $\frac{n-1}{2}$ mutual final calls, and $n-4$ miscellaneous calls. If the miscellaneous calls do not isolate a vertex, then $c_i\ge 1$ for all $i$.
\end{lemma}

\begin{proof}
The miscellaneous edges determine at least $4$ connected components as a graph. A temporal path beginning at $v_i$ or ending at $v_i$ contains miscellaneous edges from the component of $v_i$, its first caller, or its final caller. Therefore, a miscellaneous edge in a fourth component is not involved in communicating to $v_i$ or from $v_i$.
\end{proof}

Recall that $a_i+b_i\le m+2+d_i-c_i$, which turns into 
$$d_i+(n-a_i)+(n-b_i)\ge 3+c_i.$$
If it were the case that $c_i\ge 1$ for all $i$, then summing would yield $2m+(n^2-N)\ge 4n$, whence $N\le n^2-5$. Otherwise, there is a vertex $v_i$ that is isolated by miscellaneous calls, and so $d_i\le 2$.

Collecting the facts, we find a unique vertex whose first call is not mutual, say $v_1$, with $a_1=n-2$, as well as a unique vertex $v_j$ whose final call is not mutual, with $b_j<n$. Since neither of these is $v_i$, $a_i=n$, and since the same information is known by mutual final callers, $b_i=n$, unless we are done. This implies a contradiction, $2\ge 3+c_i$.

\bigskip

Theorem \ref{main} is thus established. $\qed$

\section{Conclusion}

Although the bulk of our proof built on different solutions to the Gossiping Problem, the key idea for the subcase $m=2n-5$ is the heart of a more general proof that works if calls are between $k$ people, where $k\ge 2$ is fixed \cite{KSh}. Hence, a logical next question for Greedy Gossiping is whether the same phenomenon occurs if the calls form a temporal $k$-uniform hypergraph.

\begin{conjecture}
Let $n$ dons gossip in conference calls of $k$ people. Let $m$ be a fixed number of calls, less than the minimal number of calls needed for all the dons to know all the gossip. Then the number of items they know is maximal when every call maximises the amount of new gossip overheard.
\end{conjecture}

\section{Acknowledgements}

The author is especially grateful to Imre Leader for proofreading and to Trinity College Cambridge for financial support (Internal Graduate Studentship).

\textsc{Department of Pure Mathematics and Mathematical Statistics, University of Cambridge, Wilberforce Road, Cambridge CB3 0WB.} \\ \\
\textit{E-mail address:} \texttt{kkw25@cam.ac.uk}

\end{document}